\documentclass{article}

\usepackage{graphicx}
\usepackage[T1]{fontenc}
\usepackage[latin1]{inputenc}
\usepackage{array}
\usepackage{color}

\usepackage{amsmath,amsfonts,amssymb}
\usepackage{amsthm}
\usepackage{amsfonts}
\usepackage{stmaryrd}
\usepackage{ulem}

\let\ds=\displaystyle

\let\to=\rightarrow
\newcommand{\be}{\begin{enumerate}}
\newcommand{\ee}{\end{enumerate}}
\newcommand{\bi}{\begin{itemize}}
\newcommand{\ei}{\end{itemize}}

\def \pa#1{{\left(#1\right)}}

\def \ac#1{{\left\{#1\right\}}}

\def\bs{\bigskip}
\def\ms{\medskip}

\font\ineg=msam8
\def\ie{\mathrel{\hbox{\ineg 6}}} 
\def\se{\mathrel{\hbox{\ineg >}}} 

\font\matcinq=msbm5
\font\matsept=msbm7
\font\matdix=msbm10
\newfam\matfam \scriptscriptfont\matfam=\matcinq \scriptfont\matfam=\matsept \textfont\matfam=\matdix
\def\mat{\fam\matfam}
\def\N{{\mat N}} \def\Z{{\mat Z}}
 \def\Q{{\mat Q}}
 \def\C{{\mat C}}

\newcommand{\bib}[2]{\hbox{\hbox to 12mm{[#1] :\hfill} \hfill \hbox to 144mm{\vtop{\hsize=144mm#2\vfill\ms}\hfill}}}

\let\eps=\varepsilon
\let\phi=\varphi

\newtheorem{Def}{Definition I-\!}[section]
\newtheorem{Thm}{Theorem I-\!}[section]
\newtheorem{Lem}{Lemma I-\!}[section]
\newtheorem{Cor}{Corollary I-\!}[section]
\newtheorem{Not}{Notation I-\!}[section]

\begin{document}

\begin{center}
\Large{\textsc{Structure and bases of modular space sequences}}

\Large{\textsc{$(M_{2k}(\Gamma_0(N)))_{k\in \N^*}$ and $(S_{2k}(\Gamma_0(N)))_{k\in \N^*}$}}
\bs
\ms

\large\textsc{Part I : Strong modular units}
\end{center}
\ms

\begin{center}
\large Jean-Christophe Feauveau
\footnote{Jean-Christophe Feauveau,\\
Professeur en classes préparatoires au lycée Bellevue,\\
135, route de Narbonne BP. 44370, 31031 Toulouse Cedex 4, France,\\
email : Jean-Christophe.Feauveau@ac-toulouse.fr
}

\end{center}

\begin{center}
\large JAugust 30, 2018
\end{center}

\bs

\textsc{Abstract.}

The modular discriminant $\Delta$ is known to structure the sequence of modular forms $(M_{2k}(SL_2(\Z)))_{k\in \; \N^*}$ at level $1$.
For all positive integer $N$, we define a strong modular unit $\Delta_N$ at level $N$ which enables one to structure the family $(M_{2k}(\Gamma_0(N)))_{k\in \; \N^*}$ in an identical way. We will apply this result to the bases search for each of the spaces $(M_{2k}(\Gamma_0(N)))_{k\in \; \N^*}$.
\ms

This article is the first in a series of three. In the second part we will propose explicit bases of $(M_{2k}(\Gamma_0(N)))_{k\in \; \N^*}$ for $1\ie N \ie 10$. Finally, in a third part, we will apply the results obtained in the first two parts to $(S_{2k}(\Gamma_0(N)))_{k\in \; \N^*}$.
\bs
\bs

\bs

Classification A.M.S. 2010 : 11F11, 11G16, 11F33, 33E05.
\bs
\bs

{\bf Introduction}
\ms

When studying modular forms, an important result concerns the structure of the $(M_{2k}(SL_2(\Z)))_{k \in \; \N^*}$ obtained through the $\Delta$ function, and the ability to provide an explicit basis for each subspace \cite{Serre} pages 143-144.
\ms

Such a result seems to be missing for spaces $(M_{2k}(\Gamma_0(N)))_{k \in \; \N^*}$, when $N\se 2$. We propose in this article an explicit decomposition of modular form spaces $(M_{2k}(\Gamma_0(N)))_{(k,N) \in \; \N^{*2}}$. Such a reduction cannot be simple, the formulas giving the dimension of these spaces \cite{Diam}, \cite{Stein} are a sufficient clue. However, we will show that, for a fixed $N$ level, there is a $\Delta_{N}$ function that will play for $(M_{2k}(\Gamma_0(N)))_{k\in \; \N^*}$ the role of $\Delta = \Delta_1$ in the study of $(M_{2k}(SL_2(\Z)))_{k\in \; \N^*}$.
\ms

More precisely, for any strictly positive integer $N$ and noting $\rho_N$ the weight of $\Delta_{N}$, we will establish, specifying this result :

\be
\item[]
\textit{The knowledge of $M_{2k}(\Gamma_0(N))$ bases for $1 \ie k \ie \frac{1}{2}\rho_N + $1 leads to the knowledge of $M_{2k}(\Gamma_0(N))$ bases for the entire $k$.}
\ee

\bs

In Part II, we will see how to explicitly describe $B_{2k}(N)$ bases of $M_{2k}(\Gamma_0(N))$ for $1 \ie k \ie \frac{1}{2}\rho_N + 1$ when $N$ is between 1 and 10, using Weierstrass elliptical functions.
\ms

Moreover, and for any $N$, this result is algorithmic. It enables one to obtain basic Fourier developments $(B_{2k}(N))_{k\in \; \N^*}$ to a given order of precision as soon as one has such developments for $1 \ie k \ie \frac{1}{2}\rho_N + 1$, which it is possible to obtain with the help of SAGE for example.
\bs

First, the structure of families $(M_{2k}(\Gamma_0(N)))_{k\in\; \N^*}$ will be studied under the assumption of the existence of a strong modular unit $\Delta_N$. This will be demonstrated when $N$ is a prime number, then generalized for any $N$.
\bs
\ms

\section{-- Some reminders on modular forms}

We recall some usual definitions and notations. For a rich and structured course on modular forms, we can consult \cite{Apos} or, more advanced, \cite{Diam}. We only recall what we will need afterwards.
\ms

\begin{Def}
We call Poincaré half-plane the set ${\cal H} = \{\in \C \ / \ Im(\tau) > 0\}$.
\end{Def}

In what follows, $\tau$ is a complex variable with values in $\cal H$, and we note $q = e^{2i\pi\tau}$.

\begin{Def}
For an integer $N\se 1$, we set $\Gamma_0(N) = \{\begin{pmatrix} a & b\\ c & d\end{pmatrix} \in SL_2(\Z) \ / \ c \equiv \ 0 \ {\rm mod} \ N\}$.

This is a subgroup of $SL_2(\Z)$ at level $N$ of finite index.
\end{Def}
\ms

\begin{Def}\label{Def3}
Let $k$ be an integer. A modular form of weight $k$ with respect to $\Gamma_0(N)$ is a holomorphic function $\Phi : {\cal H} \to \C$ verifying
\be
\item[(i)]
\begin{equation}
\Phi\pa{\frac{a\tau + b}{c\tau + d}} = (c\tau + d)^{k} \Phi(\tau) \ \text{for all} \ \begin{pmatrix} a & b\\ c & d\end{pmatrix}\in \Gamma_0(N). \label{MF}
\end{equation}

\item[(ii)]
For any $\begin{pmatrix} a & b\\ c & d\end{pmatrix}\in SL_2(\Z)$, $\tau \mapsto (c\tau + d)^{-k} \Phi\pa{\frac{a\tau + b}{c\tau + d}}$ admits a limit in $\C$ when $\tau$ tends towards $i\infty$, i.e. $Im(\tau)$ tends towards $+\infty$.
\ee
Note $M_{2k}(\Gamma_0(N))$ the space of the modular forms of weight $k$ with respect to $\Gamma_0(N)$.
\end{Def}
\ms

The condition $(i)$ indicates the weak modularity property of $\Phi$, whereas $(ii)$ is equivalent to the existence of a limit \textit{at all cusps}, i.e. :
\be
\item[$(ii)'$]
\textit{At the infinite cusp}, we ask for the existence of $\ds\lim_{\tau\to i\infty} \Phi(\tau + b)$ for $b\in \Z$.
\item[$(ii)''$]
\textit{For the rational cusp $\ds r = -\frac{d}{c}$ with $gcd(c,d)=1$}, let $(a,b)\in \Z^2$ be such that $\begin{pmatrix} a & b\\ c & d\end{pmatrix}\in SL_2(\Z)$, we ask for the existence of
$\ds \lim_{\genfrac{}{}{0pt}{1}{\tau\to -d/c}{\tau\in {\cal H}}} (c\tau + d)^{-k} \Phi\pa{\frac{a\tau + b}{c\tau +d}}$.

\ee

The equivalence between $(ii)$ and $[(ii)' \cup (ii)'']$ comes from the fact that $A\mapsto -A^{-1}$ is a bijection on $SL_2(\Z)$.
\ms

Of course, just check $(ii)$ on a class representative system of $SL_2(\Z)/\Gamma_0(N)$. We speak of cancellation at a cusp in the event of a zero limit at this cusp, and not of cancellation otherwise.
\bs

Thereafter, we only consider even weights $k$ because the definition I-\ref{Def3} leads to null spaces for odd weights. It is well known, and not trivial, that dimension of $M_{2k}(\Gamma_0(N))$ spaces are all finite. Note $d_{2k}(N)$ the dimension of $M_{2k}(\Gamma_0(N))$.
\bs

For $k \se 2$ and $\tau \in {\cal H}$, we define the normalized Eisenstein series \cite{Apos,Diam}:

\begin{equation}
\begin{array}{lcl}
E_{2k}(\tau) & = & \ds 1 + c_{2k}\sum_{n=1}^{+\infty} \sigma_{2k-1}(n)q^n\\
 & = & \ds \frac{1}{2\zeta(2k)} \ds \sum_{\genfrac{}{}{0pt}{1}{(m,n)\in \Z^2}{(m,n)\not = (0,0)}} \frac{1}{(m\tau+n)^{2k}}
\end{array}
\end{equation}

with $\ds c_{2k} = \frac{(2i\pi)^{2k}}{(2k-1)!\zeta(2k)}$. 
\ms

It is easy to show that $E_{2k}\in M_{2k}(SL_2(\Z))$, which ensures the non-triviality of the space. Nevertheless, it is the $\Delta\in M_{12}(SL_2(\Z))$ function that will structure the sequence $(M_{2k}(SL_2(\Z)))_{k\in\N^*}$ :
\begin{equation}
\forall \tau\in {\cal H}, \ \ \Delta(\tau) = q \prod_{n=1}^{+\infty}(1-q^n)^{24} = q-24q^2+252q^3+\ldots
\end{equation}

The $\Delta$ function is holomorphic, does not cancel on $\cal H$ and $\ds \lim_{\tau\to \infty} \Delta(\tau) = 0$, it cancels at the infinite cusp.
\ms

We recall the well-known result of structure of modular forms with respect to $SL_2(\Z) = \Gamma_0(1)$ :
\begin{equation}
\forall k\se 6, \ M_{2k}(\Gamma_0(1)) = Vect(E_{2k}) \oplus \Delta . M_{2k-12}(\Gamma_0(1)). \label{St1}
\end{equation}

Indeed, the application $\ds \Phi\mapsto \Phi .\Delta^{-1}$ is an isomorphism between the space of the modular forms of weight $2k$ which cancel at $i\infty$ (the cuspidal forms) and $M_{2k-12}(\Gamma_0(1))$ \cite{Serre}. It is this result that we will generalize.
\bs
\bs

\section{-- Spaces $\pa{M_{2k}(\Gamma_0(N))}_{k\in \N^*}$ structure}

\begin{Def}
Let $k$ and $N$ be two positive integers, and $\Phi\in M_{2k}(\Gamma_0(N))$. The function $\Phi$ is a $2k$ modular unit with respect to $\Gamma_0(N)$ (or at level $N$) if and only if :

\be
\item[(i)]
The $\Phi$ function does not cancel on $\cal H$.
\item[(ii)]
The $\Phi$ function cancels at the infinite cusp.
\item[(iii)]
The $\Phi$ function does not cancel at any other cusp with respect to $\Gamma_0(N)$.
\ee
\end{Def}

If instead of $(iii)$ one requires also cancellation at all other rational cusps, one obtains the cuspidal forms. And if the other cusps are asked not to be cancelled (condition $(iii)$), we have the notion of strong modular unity. These are two natural ways to generalize ownership of $\Delta$, which cancels at the unique cusp with respect to $\Gamma_0(1)$.
\ms

\begin{Def}\label{DefUMF}
A modular form $\Phi(\tau) = a q^n+ O(q^{n+1})$, with $a\not = 0$, is said to be of valuation $n$ and we write $\nu(\Phi) = n$. Moreover, the function $\Phi$ is unitary when $a =1$.

An upper triangular basis ${\cal B}_{2k}(\Gamma_0(N)) = (E_{2k,N}^{(r)})_{0\ie r \ie d_{2k}(N)-1}$ for space $M_{2k}(\Gamma_0(N))$ verifies $\nu(E_{2k,N}^{(r)}) < \nu(E_{2k,N}^{(r+1)})$ for all $0\ie r \ie d_{2k}(N)-2$. If the elements of ${\cal B}_{2k}(\Gamma_0(N))$ are unitary, we say that the basis is unitary upper triangular.
\end{Def}
\ms

\begin{Lem}
For all positive integers $N$ and $k$, the space $M_{2k}(\Gamma_0(N))$ admits a unitary upper triangular basis. Moreover, if $(E_{2k,N}^{(r)})_{0\ie r \ie d_{2k}(N)-1}$ is such a basis, then the sequence $(\nu(E_{2k,N}{(r)}))_{0\ie r \ie d_{2k}(N)-1}$ is independent of the basis's choice.
\end{Lem}

\begin{proof}
The existence comes directly from the Gauss process. The result on valuations is straightforward.
\end{proof}
\ms

\begin{Thm}
Let $N\in\N^*$ such as there is a strong modular unit at level $N$. Let $\Phi_0$ be such a strong modular unit at level $N$ and minimum weight $2k_0$, then the other strong modular units of the sequence $(M_{2k}(\Gamma_0(N)))_{k\in \N^*}$ are exactly of the form $\alpha \Phi_0^n$ with $\alpha\in \C^*$ and $n\in \N^*$.
\end{Thm}

\begin{proof}
Let $\Phi$ be a modular unit of weight $2k$ with $k \se k_0$.

By euclidean division $k = qk_0 + r$, $0\ie r < k_0$.
If $\nu(\Phi) < q.\nu(\Phi_0)$, then $\Phi_0^{q} \Phi^{-1} \in M_{-2r}(\Gamma_0(N))$. This function would be infinitely cancelled and would therefore be null, this is impossible.

If $q.\nu(\Phi_0) < \nu(\Phi)$, then $\Phi \Phi_0^{-q} \in M_{2r}(\Gamma_0(N))$ would be a strong modular unit, which contradicts the minimality of $k_0$. Therefore $q.\nu(\Phi_0) = \nu(\Phi)$ and $\Phi \Phi_0^{-q}$ does not cancel out on $\cal H$ or in any cusp, it is a constant non-zero modular form.
\end{proof}
\ms

The following result gives the structure of the $(M_{2k}(\Gamma_0(N)))_{k\in\N^*}$  when you have a strong modular unit (which will always be the case).
\ms

\begin{Thm}\label{ThmStruct1}
Let's say $N\in \N^*$ and $\Phi$ a $2\ell$ strong modular unit with respect to $\Gamma_0(N)$. For $k\in \N^*$, we write $(E_{2k,N}^{(s)})_{0\ie s \ie d_{2k}(N)-1}$ an upper triangular basis of $M_{2k}(\Gamma_0(N))$. So

\begin{equation}
\forall k\in \N, \ k\se \ell, \ \ M_{2k}(\Gamma_0(N)) = \Phi.M_{2k-2\ell}(\Gamma_0(N)) \oplus Vect\pa{E_{2k,N}^{(s)} \ / \ \nu(E_{2k,N}^{(s)}) < \nu(\Phi)}.
\end{equation}
Therefore, if $k\in \N^*$ and $k = q\ell+r$ with $1\ie r \ie \ell$,
\begin{equation}
M_{2k}(\Gamma_0(N)) = \Phi^{q}. M_{2r}(\Gamma_0(N)) \bigoplus_{n=0}^{q-1} \Phi^n .Vect\pa{E_{2k-2n\ell,N}^{(s)} \ / \ \nu(E_{2k-2n\ell,N}^{(s)}) < \nu(\Phi)}.
\end{equation}
\end{Thm}

\begin{proof}
As in the $N = 1$ case, the result comes from isomorphism :
\ms

\begin{tabular}{ccl}
$\phi : Vect\pa{E_{2k,N}^{(s)} \ / \ \nu(E_{2k,N}^{(s)}) \se \nu(\Phi)}$ & $\to$ & $M_{2k-2\ell}(\Gamma_0(N))$\\
$\Psi$ & $\mapsto$ & $\ds \Psi/\Phi$.
\end{tabular}

The $(iii)$ condition of the I-\ref{DefUMF} definition is essential here. For $\Psi\in M_{2k}(\Gamma_0(N))$ with $\nu(\Psi) \se \nu(\Phi)$, the quotient $\ds \Psi/\Phi$ admits a finite limit at any cusp other than infinity, which makes it possible to verify the $(ii)$ condition of definition \ref{Def3}.
\end{proof}
\bs

The objective is nevertheless to provide \textit{concrete and calculable results}. Theorem I-\ref{ThmStruct1} does not meet these criteria until we know how to calculate the elements of $\ac{E_{2k,N}^{(s)} \ / \nu(E_{2k,N}^{(s)}) < \nu(\Phi)}$, and especially until the existence of $\Phi$ is established.
\ms

To build the strong modular units, the essential tool will be Dedekind's $\eta$ function of which we recall some properties.
\bs
\ms

\section{-- Reminders on Dedekind's $\eta$ function}
\ms

With elliptical functions (Weierstrass or Jacobi), the $\eta$ function of Dedekind is the essential tool to build modular functions and forms.  Rademacher \cite{RadRam} inaugurates the construction of modular functions (of weight $0$) with respect to $\Gamma_0(p)$, $p$ prime, starting from $\eta$. But it is Newman \cite{New1}, \cite{New2} who establishes a first general result allowing to build a (weakly) modular function with respect to $\Gamma_0(N)$ starting from $\eta$. This was followed by works extending these results to the modular forms with respect to $\Gamma_0(N)$ that interest us here \cite{Ligo}, \cite{Ono} or \cite{Kohl}. The following results are, essentially, from \cite{Apos} and \cite{Kohl}.\ms

We define Dedekind's function, weight $\frac{1}{2}$, by asking \cite{Apos} :
\begin{equation*}
\forall \tau \in {\cal H}, \ \ \eta(\tau ) = e^{i\pi \tau /12}\prod_{n=1}^{+\infty}(1-q^n).
\end{equation*}

\begin{Def}
Let $N$ be a positive integer. We call $\eta$-quotient at level $N$, any function of the type
\begin{equation}
\forall \tau\in {\cal H}, \ \ \Phi(\tau) = \prod_{m\mid N} \eta(m\tau)^{a_m} \label{Poids}
\end{equation}
where $(a_1,\ldots,a_N)$ is a sequence of relative integers indexed by the positive divisors of $N$.
\end{Def}

The relationship $(\ref{Poids})$ shows that if $\Phi$ is modular, its weight is necessarily $2k = \frac{1}{2}\sum_{m\mid N} a_m$, and in this case $\nu(\Phi) = \frac{1}{24} \sum_{m\mid N} m a_m \in \N^*$.
\bs

The following results will avoid many calculations in future demonstrations. They are derived from the modular properties of the $\eta$ function and are found in various forms. The initial sources are  \cite{New1} Theorem 1,  \cite{Ligo} Proposition 3.2.1 and finally \cite{Kohl} Corollary 2.3, and \cite{Ono} Theorem 1.64.
\ms

\begin{Thm}\label{ThUMF0}
Let $\ds \Phi(\tau) = \prod_{m\mid N} \eta(m\tau)^{a_m}$ be a $N$-level $\eta$-quotient. For a divisor $m$ of $N$, we note $m' = N/m$.
If the function $\Phi$ verifies
\be
\item[(i)]
\begin{equation}
\prod_{m\mid N} m'^{a_m} \in \Q^2.
\end{equation}

\item[(ii)]
\begin{equation}
\frac{1}{24} \sum_{m\mid N} m a_m \in \Z \label{K1}
\end{equation}

\item[(iii)]
\begin{equation}
\frac{1}{24} \sum_{m\mid N} m' a_m \in \Z \label{K2}
\end{equation}
\ee
So $\Phi$ is a weakly modular form (that is, vérifies $(\ref{MF})$) with respect to $\Gamma_0(N)$ for weight $2k = \frac{1}{2}\sum_{m\mid N} a_m$.
\end{Thm}
\ms

\begin{Thm}
For $r = -\frac{d}{c}\in \Q$ with $gcd(c,d)=1$, the cancellation order of $\ds \Phi(\tau) = \prod_{m\mid N} \eta(m\tau)^{a_m}$ at the cusp $r$ is defined by

\be
\item[]
\begin{equation}
\text{ord}(\Phi,r) = \frac{N}{24} \sum_{m\mid N} \frac{gcd(c,m)^2}{m} a_m. \label{K5}
\end{equation}
\ee

The function $\Phi$ allows a limit at the cusp $r$ if and only if $\text{ord}(\Phi,r) \se 0$ and $\Phi$ cancels at this cusp if and only if $\text{ord}(\Phi,r) > 0$.\ms

Therefore, under the assumptions $(i)$, $(ii)$ and $(iii)$ of Theorem I-\ref{ThUMF0}, and if for any cusp $r = -\frac{d}{c}\in \Q$ we have $\text{ord}(\Phi,r) \se 0$, then $\Phi\in M_{2k}(\Gamma_0(N))$.
\end{Thm}
\ms

As noted in \cite{Kohl}, the behavior of $\Phi$ at the cusp $-d/c$ depends only on $c$. Since $gcd(c,m) = gcd(gcd(c,N),m)$ for any divisor $m$ of $N$, we deduce that it is enough to check the condition $\text{ord}(\Phi,r) \se 0$ at the cusps $r=1/c$ for the divisors $c$ of $N$, $1\ie c\ie N$.

The $\text{ord}(\Phi,\frac{1}{c}) = 0$ condition, for the values $1\ie c \ie N-1$, indicates the non nullity of $\Phi$ at all rational cusps. The $\text{ord}(\Phi,\frac{1}{N}) > 0$ condition indicates that $\Phi$ is cancelled at the infinite cusp, because $I_2$ and $\begin{pmatrix} 1 & 0\\ N & 1\end{pmatrix}$ are two representatives of the $\Gamma_0(N)$ class. Finally, we have the following result.
\ms

\begin{Thm}\label{ThUMF}
Let $\ds \Phi(\tau) = \prod_{m\mid N} \eta(m\tau)^{a_m}$ a $\eta$-quotient at level $N$ such that :
\be
\item[(i)]
\begin{equation}
\prod_{m\mid N} m'^{a_m} \in \Q^2 \label{UMF1}
\end{equation}

\item[(ii)]
\begin{equation}
\frac{1}{24}\sum_{m\mid N} ma_m \in \N^*.\label{UMF2}
\end{equation}

\item[(iii)]
\begin{equation}
\forall c\in \llbracket1,N-1\rrbracket, \ \ \sum_{m\mid N} \frac{gcd(c,m))^2}{m} a_m = 0\label{UMF3}
 \end{equation}
\ee

So the $\Phi$ function is a strong modular unit at level $N$ and $\ds 2k = \frac{1}{2}\sum_{m\mid N} a_m$.
\end{Thm}

\begin{proof}
For such a function $\Phi$, the condition $(ii)$ of Theorem I-\ref{ThUMF0} is deducted from the condition $(ii)$ above, and the condition $(iii)$ of Theorem I-\ref{ThUMF0} is deducted from the condition $(iii)$ above.
\ms

The condition $(ii)$ expresses the cancellation of $\Phi$ in the infinite cusp while giving the order of $\Phi$ to infinity (i.e. its valuation). Condition $(iii)$ indicates no nullity of $\Phi$ at any cusp other than the infinite cusp.
\end{proof}
\ms

We will use Theorem I-\ref{ThUMF} to construct, in paragraph $4$, a modular unit $\Delta_p$ when the level $p$ is prime. This will result in a more precise and operational version of Theorem I-\ref{ThmStruct1} in paragraph $5$. The results obtained for $p$ prime will be extended to paragraphs $6$ and $7$ at any level $N$.
\bs
\bs

\section{-- Strong modular units $\Delta_p$, $p$ prime}
\ms

Let's start by building strong unitary modular units of minimum weight for $p=$2 and $p=$3, these cases being exceptions.
\bs

\textbf{$\bullet$ The $p=2$ case.}
\ms

The space $M_2(\Gamma_0(2))$ is $1$ and generated by a form $E_{2,2}^{(0)}(\tau) = 1 + O(q)$. This is a classic result that will again be established in Part II. This excludes the existence of a strong modular form of $2$ weight that must cancel each other out in the infinite cusp.

\begin{Thm} The function
\begin{equation}
\begin{array}{lcl}
\ds \Delta_2(\tau) & = & \eta(2\tau)^{16} \eta(\tau)^{-8} \\
 & = & \ds q \prod_{k=1}^{+\infty} \frac{(1-{q}^{2k})^{16}}{(1-{q}^{k})^{8}}
\end{array}
\end{equation}
belongs to $M_4(\Gamma_0(2))$, it is a strong modular unit of minimum weight for level $2$.
\end{Thm}
\ms

\begin{proof}
The $\Delta_2$ function is a $\eta$-quotient at level $N=2$, with divisors $m\in \{1,2\}$ with $a_1 = -8$ and $a_2 = 16$.

The function $\Delta_2$ of $2k = \frac{1}{2}(a_1+a_2) = 4$ verifies the hypotheses of the application of Theorem I-\ref{ThUMF} :
\begin{equation*}
\prod_{m\mid 2} m'^{a_m} = 2^{-8} \in \Q^2, \ \ \frac{1}{24}\sum_{m\mid 2} m a_m = 1\in \N^* \ \ \text{and} \ \ \sum_{m\mid 2} \frac{a_m}{m} = 0.
\end{equation*}
\end{proof}
\bs

\textbf{$\bullet$ The $p=3$ case.}
\ms

The one dimensional space $M_2(\Gamma_0(3))$ is generated by $E_{2,3}^{(0)}(\tau) = 1+12q+O(q^2)$. There is no strong modular unit in this space.
\ms

The two dimensional $M_4(\Gamma_0(3))$ space does not contain a strong modular unit. Indeed, we can choose $E_{4,3}^{(0)} =[E_{2,3}^{(0)}]^2 = 1+24q + O(q^2)$, but we also know an element of $M_2(\Gamma_0(3))$ built from the series of Eisenstein $E_4$, namely $E_4(3\tau) = 1+240q^3 + O(q^6)$.
\ms

We deduce from these two linearly independent modular forms that $E_{4,3}^{(1)}$ is of valuation $1$ (and unique if unitary).
This function could be a strong modular unit, but using division, we should have $\dim(M_6(\Gamma_0(3))) = 2$ which is wrong, the space is $3$. We then have the following result.

\begin{Thm}
The function
\begin{equation}
\begin{array}{lcl}
\Delta_3(\tau) & = & \ds \eta(3\tau)^{18} \eta(\tau)^{-6}\\
  & = & \ds q^{2} \prod_{k=1}^{+\infty} \frac{(1-q^{3k})^{18}}{ (1-q^{k})^{6}}
\end{array}
\end{equation}
belong to $M_6(\Gamma_0(3))$, it is the strong modular unit of minimum weight for the level $3$.
\end{Thm}

\begin{proof}
The $\Delta_3$ function is an $\eta$-quotient at level $N=3$, divisors $m\in \{1,3\}$ with $a_1 = -6$ and $a_2 = 18$.

The $\Delta_3$ function of weight $2k = \frac{1}{2}(a_1+a_3) = 6$ verifies the hypotheses of application of Theorem I-\ref{ThUMF} :
\begin{equation*}
\prod_{m\mid 3} m'^{a_m} = 3^{-6} \in \Q^2, \ \ \text{and} \ \ \frac{1}{24}\sum_{m\mid 3} m a_m = 2\in \N^*.
\end{equation*}
Finally, for $c=1$ and $c=2$, we find $\ds \sum_{m\mid 3} \frac{gcd(c,m))^2}{m} a_m = \sum_{m\mid 3} \frac{a_m}{m} = 0$.
\end{proof}
\bs

\textbf{$\bullet$ The $p\se 5$ case, $p$ prime.}
\ms

We then have a general shape for a strong modular unit of $p\se5$, $p$ prime.

\begin{Thm} For any prime number $p \se 5$, we define on $\cal H$ the function $\Delta_p$ :
\begin{equation}
\begin{array}{lcl}
\Delta_p(\tau) & = & \ds \eta(p\tau)^{2p} \eta(\tau)^{-2}\\
 & = & \ds {q}^{({p}^{2}-1)/12}\prod _{n=1}^{+\infty}{\frac{(1-{q}^{pn}) ^{2p}}{ (1-{q}^{n})^{2}}}. \label{DeltaEta}
\end{array}
\end{equation}
So $\Delta_p$ is a strong modular unit of $M_{p-1}(\Gamma_0(p))$.
\end{Thm}
\ms

Note the equality $\Delta_p(\tau ^{12}=\Delta(p\tau )^{p}\Delta(\tau ^{-1}$ which indicates a modular property for the $p-1$ weight function $\Delta_p$.
\ms

Let's also note that if $p\se 5$ is prime, then $\frac{p^2-1}{12}\in \N$.
More generally, if $N$ is a positive integer such as $N \equiv 1 (\text{mod} \ 6)$ or $N \equiv 5 (\text{mod} \ 6)$, which is the case for $p\se 5$ prime, then $frac{N^2-1}{12} \in \N$.

Indeed, if $N = 6k+1$ then $\frac{N^2-1}{12} = 3k^2+k$ and if $N = 6k+5$ then $\frac{N^2-1}{12} = 3k^2+k+2$.
\ms

\begin{proof}
The function $\Delta_p$ is an $\eta$-quotient at level $p$, with divisors $m\in \{1,p\}$ and $a_1 = -2$ and $a_p = 2p$.

The function $\Delta_p$ has weight $2k = p-1\in 2\N^*$ and verifies the hypotheses of application of Theorem I-\ref{ThUMF} :
\begin{equation*}
\prod_{m\mid p} m'^{a_m} = p^{-2} \in \Q^2, \ \ \text{and} \ \ \frac{1}{24}\sum_{m\mid p} m a_m = \frac{p^2-1}{12}\in \N^*.
\end{equation*}
Finaly, for $c\in \llbracket1,p-1\rrbracket$, we have $\ds \sum_{m\mid p} \frac{gcd(c,m))^2}{m} a_m = \sum_{m\mid p} \frac{a_m}{m} = 0$.
\end{proof}
\bs
\bs

\section{-- Structure and bases of $\pa{M_{2k}(\Gamma_0(p))}_{k\in \N^*}$, $p$ prime}
\ms

We are looking for building a sequence of unitary upper triangular bases $({\cal B}_{2k}(\Gamma_0(p)))_{k\in\N^*}$ of $\pa{M_{2k}(\Gamma_0(p))}_{k\in \N^*}$ and we note in a generic way ${\cal B}_{2k}(\Gamma_0(p)) = (E_{2k,p}^{(s)})_{0\ie s \ie d_{2k}(p)-1}$.
\ms

Let's start with the special $p=2$ and $p=3$ cases that must be dealt with separately. This also makes it possible to understand the algorithm of production of the bases on these two examples.
\bs

\textbf{$\bullet$ The $p=2$ case.}
\ms

Let $E_{2,2}^{(0)}$ be the unitary generator of $M_{2}(\Gamma_0(2))$ which has valuation $0$. It is therefore possible to choose $E_{2k,2}^{(0)} =[E_{2,2}^{(0)}]^k$ as the first vector of the ${\cal B}_{2k}(\Gamma_0(2))$ unitary upper triangular bases.

The $\Delta_2$ function being of weight $4$ and valuation $1$, Theorem I-\ref{ThmStruct1} gives the following result.
\begin{Cor}
\begin{equation*}
\forall k\se 3, \ \ M_{2k}(\Gamma_0(2)) = Vect(E_{2k,2}^{(0)})\oplus \Delta_2 . M_{2k-4}(\Gamma_0(2))
\end{equation*}
moreover
\begin{equation*}
{\cal B}_{2k}(\Gamma_0(2)) = \pa{[E_{2,2}^{(0)}]^a \Delta_2^b, \ \text{with} \ (a,b)\in\N^2 \ \text{such \ as} \ a+2b = k}
\end{equation*}

is an upper triangular basis of $M_{2k}(\Gamma_0(2))$
\end{Cor}
\ms

\begin{proof}
These are direct consequences of Theorem I-\ref{ThmStruct1}.
\end{proof}
\ms

Of course, a similar result is true for $N = 1$ and leads to upper triangular bases structured by $\Delta$, instead of the usual result obtained with the $E_4$ and $E_6$ generators. The details of this will be studied in Part II.
\bs

\textbf{$\bullet$ The $p=3$ case.}
\ms

The strong modular form $\Delta_3$ is of weight $6$, valuation $2$, and Theorem I-\ref{ThmStruct1} is written

\begin{Cor}
\begin{equation}
\forall k \se 4, \ \ M_{2k}(\Gamma_0(3)) = Vect(E_{2k,3}^{(0)},E_{2k,3}^{(1)}) \oplus \Delta_3.M_{2k-6}(\Gamma_0(3)).\label{D3}
\end{equation}

We then have a basis of $M_{2k}(\Gamma_0(3))$ :

\begin{equation}
{\cal B}_{2k}(\Gamma_0(3)) = \pa{[E_{2,3}^{(0)}]^a.\Delta_3^b, \ (a,b)\in\N^2 \ / \ a+3b = k} \cup \pa{E_{4,3}^{(1)}.[E_{2,3}^{(0)}]^a.\Delta_3^b, \ (a,b)\in\N^2 \ / \ a+3b = k-2}. \label{B3}
\end{equation}

\end{Cor}

\begin{proof}
This time again, the first equality is an application of Theorem I-\ref{ThmStruct1}

We have $\dim(M_{2}(\Gamma_0(3))) = 1$, $\dim(M_{4}(\Gamma_0(3))) = 2$ with $\nu(E_{2,3}^{(0)}) = 0$, $\nu(E_{4,3}^{(0)}) = 0$ and $\nu(E_{4,3}^{(1)}) = 1$.

We can therefore choose $E_{4,3}^{(0)} = [E_{2,3}^{(0)}]^2$, and more generally, $E_{2k,3}^{(0)} =[E_{2,3}^{(0)}]^k$ as the first element of the unitary upper triangular basis of $M_{2k}(\Gamma_0(3))$.

Likewise, it is legitimate to choose for everything $k\se3$, $E_{2k,3}^{(1)} = E_{4,3}^{(1)}[E_{2,3}^{(0)}]^{k-2}$.
\ms

It is easy to verify that the $(\ref{B3})$ relationship produces a basis for $k=1$ and $k=2$, we assume the result true up to the order $k-1\se 2$. Based on the above, the $(\ref{D3})$ relationship indicates that $([E_{2,3}^{(0)}]^k, E_{4,3}^{(1)}[E_{2,3}^{(0)}]^{k-2}) \cup \Delta_3 {\cal B}_{2k-4}(\Gamma_0(3))$ provides a basis ${\cal B}_{2k}(\Gamma_0(3))$.
\ms

It can then be seen that

$\ds \pa{[E_{2,3}^{(0)}]^a.\Delta_3^b, \ (a,b)\in\N^2 \ / \ a+3b = k} = ([E_{2,3}^{(0)}]^k) \cup \Delta_3\pa{[E_{2,3}^{(0)}]^a.\Delta_3^b, \ (a,b)\in\N^2 \ / \ a+3b = k-3}$

and

$\ds \pa{E_{4,3}^{(1)}.[E_{2,3}^{(0)}]^a.\Delta_3^b, \ (a,b)\in\N^2 \ / \ a+3b = k-2} =$

{\hskip 5cm} $\ds E_{4,3}^{(1)}[E_{2,3}^{(0)}]^{k-2} \cup \Delta_3\pa{E_{4,3}^{(1)}.[E_{2,3}^{(0)}]^a.\Delta_3^b, \ (a,b)\in\N^2 \ / \ a+3b = k-5}$

which concludes the recurrence demonstration.
\end{proof}
\ms

\textbf{$\bullet$ The $p\se 5$ case, $p$ prime.}
\ms

We set $p\se 5$, $p$ prime.
\ms

\begin{Lem}\label{LemDim1}
for every $k\in\N^*$,
\begin{equation}
\ds \dim(M_{2k+p-1}(\Gamma_0(p))) - \dim(M_{2k}(\Gamma_0(p))) = \nu(\Delta_p) = \frac{p^2-1}{12}.
\end{equation}
\end{Lem}

\begin{proof}
The second equality is known, and the first is in fact a special case of Theorem I-\ref{LemDim2}, valid in any general case $N$, and which will be demonstrated in paragraph $7$.

We rely on an explicit formula giving the dimension of $M_{2k}(\Gamma_0(N))$ depending on $k$ and $N$ \cite{Stein}.
\end{proof}

Moreover, from Theorem I-\ref{ThmStruct1}, we derive the following equality
\[\forall k\in \N^*, \ \ \dim(M_{2k+p-1}(\Gamma_0(p))) = \dim(M_{2k}(\Gamma_0(p))) + Card(\{s \ / \ \nu(E_{2k+p-1,N}^{(s)}) < \frac{p^2-1}{12}\}).\]

As a result
$\ds Card(\{s \ / \ \nu(E_{2k+p-1,N}^{(s)}) < \frac{p^2-1}{12}\}) = \frac{p^2-1}{12}$, from which the following theorem is derived.

\begin{Thm} Let $p \se 5$ be a prime number and, for an integer $k\se 1$, $(E_{2k,p}^{(s)})_{0\ie s \ie d_{2k}(N)-1}$ an upper triangular basis of $M_{2k}(\Gamma_0(p))$. So
\begin{equation}
\forall k\se \frac{p+1}{2}, \ \ \forall s\in \llbracket0,\frac{p^2-1}{12} -1\rrbracket, \ \ \nu(E_{2k,p}^{(s)}) = s.
\end{equation}
\end{Thm}
\bs 

This result is important. It indicates that the new items appearing in ${\cal B}_{2k}(\Gamma_0(p))$ are regularly upper triangular, the other items coming from $\Delta_p.{\cal B}_{2k-(p-1)}(\Gamma_0(p))$. It remains to characterize these new elements.
\ms

We have the following result, true for the whole $N\se 2$.

\begin{Thm}\label{ThmVal1}
Either an integer $N\se 2$, then $M_2(\Gamma_0(N))$ has elements of valuation $0$.
\end{Thm}

\begin{proof}
It is a well known result obtained in general thanks to the false series of Eisenstein $G_2$ (see \cite{Miya} or \cite{Diam}). We define

\begin{equation*}
\begin{array}{lcl}
G_2(\tau) & = & \ds \sum_{m\in \Z}\sum_{n\in \Z'_m} \frac{1}{(m\tau+n)^{2}}\\
& = & \ds 2\zeta(2) - 8\pi^2\sum_{n=1}^{+\infty} \sigma(n) q^n
\end{array}
\end{equation*}
where $\Z'_m = \Z-\{0\}$ if $m=0$ and $\Z'_m = \Z$ otherwise.

Then we show that $G_{2,N}(\tau) = G_2(\tau)-NG_2(N\tau)$ belongs to $M_2(\Gamma_0(N))$. Moreover, $\ds \lim_{\tau\to +\infty} G_{2,N}(\tau) = 2(1-N)\zeta(2) \not= 0$, which ensures results.
\ms

A different demonstration of this result will be given in Part II.
We will show precisely that the function $\sum_{k=1}^{N-1}\wp(k\tau,N\tau)$ is an element of $M_2(\Gamma_0(N))$ valuation $0$, where $\wp(z,\tau)$ designates the elliptical function of Weierstrass on a network of periods $\Z + \tau \Z$, taken at point $z$.
\end{proof}
\ms

\begin{Cor}\label{CorVal1}
let $N\se 2$ be an integer, if $(E_{2k,p}^{(s)})_{0\ie s \ie d_{2k}(N)-1}$ is a unitary upper triangular basis of $M_{2k}(\Gamma_0(N))$, then $\nu(E_{2k,p}^{(0)}) = 0$ and we can choose $E_{2k,p}^{(0)} = [E_{2,p}^{(0)}]^k$.
\end{Cor}
\ms

The Theorem I-\ref{ThmVal1} and corollary I-\ref{CorVal1} allow the algorithmic construction of structured bases. Indeed, for $\ds k\se\frac{p+1}{2}$, we can choose

\begin{equation}
E_{2k,p}^{(s)} = E_{p+1,p}^{(s)}[E_{2,p}^{(0)}]^{k-\frac{p+1}{2}}, \ \ 0\ie s < \frac{p^2-1}{2}.
\end{equation}
These elements are unitary and regularly upper triangular (without jumps) in space $M_{2k}(\Gamma_0(p))$, so they qualify as $\ds \frac{p^2-1}{2}$ first elements of ${\cal B}_{2k}(\Gamma_0(p))$.
\ms

We can then specify Theorem I-\ref{ThmStruct1}

\begin{Thm}
Let $p \se 5$ be a prime number, then
\begin{equation}
\forall k\in \N, \ k\se \frac{p-1}{2}, \ \ M_{2k}(\Gamma_0(p)) = \Delta_p.M_{2k-(p-1)}(\Gamma_0(p)) \oplus Vect\pa{E_{p+1,p}^{(s)}[E_{2,p}^{(0)}]^{k-\frac{p+1}{2}} \ / \ 0\ie s < \frac{p^2-1}{12}}.
\end{equation}
Therefore, if $k\in \N^*$ and $k = q\frac{p-1}{2}+r$ with $1\ie r \ie \frac{p-1}{2}$,
\begin{equation}
 M_{2k}(\Gamma_0(p)) = \Delta_p^{q}. M_{2r}(\Gamma_0(p)) \bigoplus_{n=0}^{q-1}  \Delta_p^n .Vect\pa{E_{p+1,p}^{(s)}[E_{2,p}^{(0)}]^{k-(n+1)\frac{p-1}{2}-1} \ / \ 0\ie s < \frac{p^2-1}{12}}.
\end{equation}
\end{Thm}

Regarding the calculation of a unitary upper triangular basis ${\cal B}_{2k}(\Gamma_0(p))$, $k\se 1$, Theorem I-\ref{ThmStruct1} is now operational since the knowledge of all bases is reduced to that of the finite sequence of bases $({\cal B}_{2k}(\Gamma_0(p)))_{1\ie k \ie \frac{p+1}{2}}$. We will see in part II that this is possible using elliptical functions for $1\ie N\ie 10$.
\bs
\bs

\section{-- Strong modular units $\Delta_N$, $N$ positive integer}
\ms

In the previous paragraph, we obtained structured bases of $(M_{2k}(\Gamma_0(p)))_{k\in\N^*}$ when $p$ is prime. The important tool, which reduces the search from an infinity of bases to a finite number, is the existence of a strong $\Delta_p$ modular form. We will establish the existence of $\Delta_N$ in the general $N\se 1$ case.
\ms

There are sub-cases to consider. For example, the existence of $\Delta_{p^r}$ when $p$ is prime, of $\Delta_{p_1^{r_1} p_2^{r_2}}$ for $p_1$, $p_2$ different prime numbers $\ldots$ \ As one might expect, it is sometimes necessary to distinguish cases of prime factors equal to $2$ or $3$, and the case greater than $5$.
\ms

We will note in this part $\Delta_N$ a strong modular unit of the sequence $(M_{2k}(\Gamma_0(N)))_{k\in\N^*}$ while remaining consistent with the notation introduced in section 4.
\ms

Let's summarize, a little ahead of time, the knowledge about the $(\Delta_N)_{1\ie N\ie 10}$ functions (with minimum weight) that will be obtained in Part II :

\begin{equation}
\begin{array}{lcl}
\Delta_2(\tau) & = & \eta(\tau)^{-8} \eta(2\tau)^{16} = \ds q\prod_{k=1}^{+\infty} \frac{( 1-{q}^{2k})^{16}}{(1-{q}^{k}) ^{8}}\\
\Delta_3(\tau) & = & \ds \eta(\tau)^{-6} \eta(3\tau)^{18}
  = \ds {q}^{2}\prod _{k=1}^{+\infty} \frac{(1-{q}^{3k})^{18}}{(1-{q}^{k})^{6}}\\
\Delta_4(\tau) & = & \ds \eta(2\tau)^{-4} \eta(4\tau)^{8} = \ds q\prod_{k=1}^{+\infty} \frac{(1-{q}^{4n})^{8}}{(1-{q}^{2n})^{4}}\\
\Delta_5(\tau) & = & \ds \eta(\tau)^{-2} \eta(5\tau)^{10} = \ds {q}^{2}\prod_{n=1}^{+\infty} \frac{( 1-{q}^{5n})^{10}}{(1-{q}^{n})^{2}}\\
\Delta_6(\tau) &  = & \ds \eta(\tau)^2 \eta(2\tau)^{-4} \eta(3\tau)^{-6} \eta(6\tau)^{12} = \ds q^2\prod_{k=1}^{+\infty} \frac{(1-q^{k})^{2}(1-q^{6k})^{12}} {(1-q^{2k})^{4} (1-q^{3k})^{6}}\\
\Delta_7(\tau) & = & \ds \eta(\tau)^{-2} \eta(7\tau)^{14} = \ds q^4 \prod_{n=1}^{+\infty} \frac{(1-q^{7n})^{14}}{(1-q^n)^2}\\
\Delta_8(\tau) & = & \ds \eta(4\tau)^{-4} \eta(8\tau)^{8} = \ds q^2\prod_{k=1}^{+\infty} \frac{(1-{q}^{8n})^{8}}{(1-{q}^{4n})^{4}}\\
\Delta_9(\tau) & = & \ds \eta(3\tau)^{-2} \eta(9\tau)^{6} = \ds {q}^{2}\prod_{k=1}^{+\infty} \frac{(1-{q}^{9k})^{6}}{(1-{q}^{3k})^{2}}\\
\Delta_{10}(\tau) & = & \ds \eta(\tau)^2 \eta(2\tau)^{-4} \eta(5\tau)^{-10} \eta(10\tau)^{20} = \ds q^6 \prod_{k=1}^{+\infty} \frac{(1-{q}^{n})^{2} (1-{q}^{10n})^{20}}{(1-{q}^{2n})^{4} (1-{q}^{5n})^{10}}\\
\end{array}
\end{equation}
\ms

Note that $\ds \Delta_4(\tau) = \Delta_2(2\tau)^{1/2}$, $\ds \Delta_8(\tau) = \Delta_4(2\tau)$ and $\ds \Delta_9(\tau) = \Delta_3(3\tau)^{1/3}$.
\ms

The following result reduces the search for a strong modular unit at level $N$ when $N$ is a product of separate primes.

\begin{Lem}\label{Dilat}
Let $N$, $n$ be positive integers and $\Psi\in M_{2k}(\Gamma_0(N))$. Then the function $\Psi_n(\tau) = \Psi(n\tau)$ belongs to space
$M_{2k}(\Gamma_0(nN))$.
\end{Lem}

For example, a demonstration can be found in \cite{Diam}, exercise 1.2.12. The following result can be deduced.

\begin{Cor}\label{Dilat2}
Let $p_1,\ldots,p_n$ be all distinct prime numbers and $r_1,\ldots,r_n$ positive integers.

If the $\Phi(\tau)$ function is a $2k$ strong modular unit with respect to $\Gamma_0(p_1\ldots p_n)$, then $\Phi(p_1^{r_1-1}\ldots p_n^{r_n-1} \tau)$ is a $2k$ strong modular unit with respect to $\Gamma_0(p_1^{r_1}\ldots p_n^{r_n})$.
\end{Cor}

\begin{proof}
According to the lemma I-\ref{Dilat}, if $\tilde{\Phi}(\tau) = \Phi(p_1^{r_1-1}\ldots p_n^{r_n-1} \tau)$, then $\tilde{\Phi} \in M_{2k}(\Gamma_0(p_1^{r_1}\ldots p_n^{r_n}))$.

Besides, $\tilde{\Phi}$ does not cancel on ${\cal H}$ and Fourier serial development of $(c\tau+d)^{-2k}\Phi(\frac{a\tau+b}{c\tau+d})$ indicates that the $\tilde{\Phi}$ order at a given cusp is null if and only if the $\Phi$ order in that same cusp is also null, which gives the result.
\end{proof}
\ms

Following the definition $(\ref{DeltaEta})$ of $\Delta_p$ for $p\se 5$ prime, we ask

\begin{Not}
For all $q\in \N^*$,
\begin{equation}
\forall \tau \in {\cal H}, \ \ \eta_q(\tau) = \eta(q\tau)^q.
\end{equation}
\end{Not}

\bi
\item[$\bullet$] \textbf{The $N = p^r$ case, $p$ prime and $r > 0$}
\ms

The result is as follows:
\ms

\begin{Thm}\label{Thp1}
The following $\Delta_{p^r}$ functions are strong modular units with respect to $\Gamma_0(p^r)$.

When $p=2$,
\begin{equation}
\Delta_2 = \pa{\frac{\eta_2}{\eta_1}}^8, \ \ \text{and} \ \ \forall r\se 2, \ \ \Delta_{2^r}(\tau) = \Delta_4(2^{r-2}\tau) = \pa{\frac{\eta_2}{\eta_1}}^4 (2^{r-1}\tau) \in M_{2}(\Gamma_0(2^r)).
\end{equation}

When $p=3$,
\begin{equation}
\Delta_3 = \pa{\frac{\eta_3}{\eta_1}}^6, \ \ \text{and} \ \ \forall r\se 2, \ \ \Delta_{3^r}(\tau) = \Delta_9(3^{r-2}\tau) = \pa{\frac{\eta_3}{\eta_1}}^2 (3^{r-1}\tau) \in M_{2}(\Gamma_0(3^r)).
\end{equation}

When $p\se 5$ prime,
\begin{equation}
\forall r\in \N^*, \ \ \Delta_{p^r}(\tau) = \Delta_p(p^{r-1}\tau) = \pa{\frac{\eta_p}{\eta_1}}^2 (p^{r-1}\tau)\in M_{p-1}(\Gamma_0(p^r)).
\end{equation}
\end{Thm}
\ms

\begin{proof}
The $\Delta_p$ case, $p$ prime was processed. According to corollary I-\ref{Dilat2}, it is enough to verify that $\Delta_4$ and $\Delta_9$ are strong modular units of $M_{2}(\Gamma_0(4))$ and $M_{2}(\Gamma_0(9))$ respectively.
\ms

The $\Delta_4(\tau) = \ds \eta(2\tau)^{-4} \eta(4\tau)^{8}$ function is an  $\eta$-quotient at level $N=4$, with divisors $m\in \{1,2,4\}$ with $a_1 = 0$, $a_2 = -4$, $a_4 = 8$ and weight $2k = \frac{1}{2}(a_1+a_2+a_4) = 2$.

It verifies the hypotheses of the application of Theorem I-\ref{ThUMF} :
\begin{equation*}
\prod_{m\mid 4} m'^{a_m} = 2^{-4} \in \Q^2, \ \ \frac{1}{24}\sum_{m\mid 4} m a_m = 1\in \N^* \ \ \text{and} \ \ \sum_{m\mid 4} \frac{a_m}{m} = 0.
\end{equation*}
The last equality gives the condition $(iii)$ of Theorem for $c=1$ and $c=3$, the $c=2$ case comes from the calculation
\begin{equation*}
\sum_{m\mid 4} \frac{gcd(2,m)^2}{m} a_m = 0.
\end{equation*}

The $\Delta_9(\tau) = \ds \eta(3\tau)^{-2} \eta(9\tau)^{6}$ function is an $\eta$-quotient at level $N=9$, with divisors $m\in \{1,3,9\}$ and $a_1 = 0$, $a_3 = -2$, $a_9 = 6$ of weight $2k = \frac{1}{2}(a_1+a_3+a_9) = 2$.

It verifies the hypotheses of the application of Theorem I-\ref{ThUMF} :
\begin{equation*}
\prod_{m\mid 9} m'^{a_m} = 3^{-2} \in \Q^2, \ \ \frac{1}{24}\sum_{m\mid 9} m a_m = 2\in \N^* \ \ \text{and} \ \ \sum_{m\mid 9} \frac{a_m}{m} = 0.
\end{equation*}
The last equality gives the condition $(iii)$ of Theorem for $c\in\{1,2,4,5,7,8\}$, the $c=3$ and $c=6$ cases come from the calculations
\begin{equation*}
\sum_{m\mid 9} \frac{gcd(3,m)^2}{m} a_m = \sum_{m\mid 9} \frac{gcd(6,m)^2}{m} a_m = 0.
\end{equation*}
\end{proof}
\bs

\item[$\bullet$] \textbf{The $N = p_1^{r_1}p_2^{r_2}$ case, with $p_1, p_2$ distinct primes and $(r_1,r_2)\in \N^{*2}$}
\ms

The guide to understanding the construction of $\Delta_{p_1^{r_1}p_2^{r_2}}$ is based on equality

\begin{equation}
(p_1-1)(p_2-1) = +p_1p_2 - p_1 - p_2 + 1 \label{dec2}
\end{equation}
which will be generalized later.
\ms

The algorithm to build $\Delta_N$ is as follows.
\be
\item
When $N = p_1.p_2$, terms with $+$ indicate 'elementary units' $\eta_q$ placed in the numerator of the strong modular unit and terms preceded by the $-$ sign indicate 'elementary units' $\eta_q$ placed in the denominator.
\item
When $N = p_1^{r_1}.p_2^{r_2}$, corollary I-\ref{Dilat2} applies and $\Delta_{p_1^{r_1}.p_2^{r_2}}(\tau) = \Delta_{p_1p_2}(p_1^{r_1-1}.p_2^{r_2-1}\tau)$.
\ee
\ms

We have the following generic result:
\ms

\begin{Thm}\label{Thp1p2}
$ $

\be
\item[]
Let $p\se 3$ be a prime and $(r_1,r_2)\in \N^{*2}$:
\begin{equation}
\Delta_{2^{r_1}p^{r_2}}(\tau) = \pa{\frac{\eta_1\, \eta_{2p}} {\eta_{2}^{2}\, \eta_{p}}}^2(2^{r_1 -1} p^{r_2 -1} \tau) \in M_{(p-1)}(\Gamma_0(2^{r_1}p^{r_2})).
\label{2p}
\end{equation}

\item[]
Let $p_1\se 3$ and $p_2\se 3$ be two distinct primes and $(r_1,r_2)\in \N^{*2}$:

\begin{equation}
\Delta_{p_1^{r_1}p_2^{r_2}}(\tau) = \frac{\eta_1\, \eta_{p_1p_2}} {\eta_{p_1}\, \eta_{p_2}}(p_1^{r_1 -1} p_2^{r_2 -1} \tau) \in M_{\frac{1}{2}(p_1-1)(p_2-1)} (\Gamma_0(p_1^{r_1}p_2^{r_2})).
\label{p1p2}
\end{equation}
\ee

These functions are strong modular units of the corresponding modular spaces.
\end{Thm}
\ms

\begin{proof}
Given the corollary I-ref{Dilat2}, it is to show that $\Delta_{2p}$ and $\Delta_{p_1p_2}$ are strong modular units in $M_{(p-1)}(\Gamma_0(2p))$ and $M_{\frac{1}{2}(p_1-1)(p_2-1)} (\Gamma_0(p_1p_2))$ respectively.
\ms

The function
\begin{equation*}
\Delta_{2p}(\tau) = \ds \eta(\tau)^{2} \eta(2\tau)^{-4} \eta(p\tau)^{-2p} \eta(2p\tau)^{4p}
\end{equation*}
is an $\eta$-quotient at level $2p$, with divisors $m\in \{1,2,p,2p\}$ and $a_1 = 2$, $a_2 = -4$, $a_p = -2p$, $a_{2p} = 4p$ with weight $2k = \frac{1}{2}(a_1+a_2+a_p+a_{2p}) = p-1$.

It verifies the hypotheses of application of Theorem I-\ref{ThUMF} :
\begin{equation*}
\prod_{m\mid 2p} m'^{a_m} = 2^{-2(p-1)} p^{-2} \in \Q^2, \ \ \frac{1}{24}\sum_{m\mid 2p} m a_m = \frac{p^2 - 1}{4}\in \N^* \ \ \text{et} \ \ \sum_{m\mid 2p} \frac{a_m}{m} = 0.
\end{equation*}

The last equality gives the condition $(iii)$ of Theorem I-\ref{ThUMF} for $c\in\llbracket1,2p-1\rrbracket-\{2,p\}$, the $c=2$ and $c=p$ cases come from the calculations
\begin{equation*}
\sum_{m\mid 2p} \frac{gcd(2,m)^2}{m} a_m = \sum_{m\mid 2p} \frac{gcd(p,m)^2}{m} a_m = 0.
\end{equation*}

Note that the square root of this function does not verify the $(i)$ condition of Theorem I-\ref{ThUMF}.
\ms

When $p_1\se 3$ and $p_2\se 3$ are distinct primes, the function
\begin{equation*}
\Delta_{p_1p_2}(\tau) = \ds \eta(\tau)^{1} \eta(p_1\tau)^{-p_1} \eta(p_2\tau)^{-p_2} \eta(p_1p_2\tau)^{p_1p_2}
\end{equation*}
is an $\eta$-quotient at level $N=p_1p_2$, with divisors $m\in \{1,p_1,p_2,p_1p_2\}$ with $a_1 = 1$, $a_{p_2} = -p_1$, $a_{p_2} = -p_2$, $a_{p_1p_2} = p_1p_2$ and of weight $2k = \frac{1}{2}(a_1+a_{p_1}+a_{p_2}+a_{p_1p_2}) = \frac{1}{2}(p_1-1)(p_2-1)$.

It verifies the hypotheses of the application of Theorem I-\ref{ThUMF} :
\begin{equation*}
\prod_{m\mid p_1p_2} m'^{a_m} = p_1^{-(p_2-1)} p_2^{-(p_1-1)} \in \Q^2, \ \ \frac{1}{24}\sum_{m\mid p_1p_2} m a_m = \frac{(p_1^2 - 1)(p_2^2-1)}{24}\in \N^* \ \ \text{and} \ \ \sum_{m\mid p_1p_2} \frac{a_m}{m} = 0.
\end{equation*}

This time again, the last equality gives the condition $(iii)$ of Theorem for $c\in\llbracket1,p_1p_2-1\rrbracket-\{p_1,p_2\}$, the $c=p_1$ and $c=p_2$ cases come from direct calculations
\begin{equation*}
\sum_{m\mid p_1p_2} \frac{gcd(p_1,m)^2}{m} a_m = \sum_{m\mid p_1p_2} \frac{gcd(p_2,m)^2}{m} a_m = 0.
\end{equation*}\ms
\end{proof}
\bs

The result can be unified by proposing $\ds \pa{\frac{\eta_1\, \eta_{p_1p_2}}{\eta_{p_1}\, \eta_{p_2}}}^2(p_1^{r_1 -1} p_2^{r_2 -1} \tau) \in M_{(p_1-1)(p_2-1)} (\Gamma_0(p_1^{r_1}p_2^{r_2}))$ for strong modular unit for all prime numbers $p_1 \not= p_2$. For all that,
The $(\ref{p1p2})$ relationship enables one to halve the weight of the strong modular unit retained when $2$ is not one of the primary factors, which will be useful when searching for bases for example.
\ms

For the $(\ref{2p})$ relationship, the valuation of $\Delta_{2^{r_1}p^{r_2}}$ is:

\begin{equation}
\nu(\Delta_{2^{r_1}p^{r_2}}) = 2^{r_1-3}p^{r_2-1} (p^{2}-1),
\end{equation}

for the $(\ref{p1p2})$ relationship, the valuation of $\Delta_{p_1^{r_1}p_2^{r_2}}$ is:

\begin{equation}
\nu(\Delta_{p_1^{r_1}p_2^{r_2}}) = p_1^{r_1-1}p_2^{r_2-1} \frac{(p_1^{2}-1)(p_2^{2}-1)}{24}
\end{equation}

which are always integers.
\bs

Let us give some examples.
\be
\item[-]
For $N = 3.5 = 15$,
\begin{equation}
\begin{array}{lcl}
\Delta_{15}(\tau) & = & \ds \frac{\eta(\tau) \eta(15\tau)^{15}}{\eta(3\tau)^{3} \eta(5\tau)^{5}} \in M_{4}(\Gamma_0(15))\\
& = & \ds q^{8}\prod_{n=1}^{+\infty} (1-q^{n}) (1-q^{3n})^{-3} (1-q^{5n})^{-5} (1-q^{15n})^{15}.
\end{array}
\end{equation}

\item[-]
For $N = 5.7 = 35$,
\begin{equation}
\begin{array}{lcl}
\Delta_{35}(\tau) & = & \ds \frac{\eta(\tau) \eta(35\tau)^{35}}{\eta(5\tau)^{5} \eta(7\tau)^{7}} \in M_{12}(\Gamma_0(35))\\
& = & \ds q^{48}\prod_{n=1}^{+\infty} (1-q^{n}) (1-q^{5n})^{-5} (1-q^{7n})^{-7} (1-q^{35n})^{35}.
\end{array}
\end{equation}

\item[-]
For $N = 2^2.3^2 = 36$,
\begin{equation}
\begin{array}{lcl}
\Delta_{36}(\tau) & = & \ds \frac{\eta(6\tau)^2 \eta(36\tau)^{12}}{\eta(12\tau)^{4} \eta(18\tau)^{6}} \in M_{2}(\Gamma_0(36))\\
& = & \ds q^{12}\prod_{n=1}^{+\infty} (1-q^{6n})^{2} (1-q^{12n})^{-4} (1-q^{18n})^{-6} (1-q^{36n})^{12}.
\end{array}
\end{equation}
\ee
\bs
\ms

\item[$\bullet$] \textbf{The $N = p_1^{r_1}p_2^{r_2}p_3^{r_3}$ case, $p_1 < p_2 < p_3$ distinct primes, $(r_1,r_2,r_3)\in \N^{*3}$}
\ms

Let's extend equality $(\ref{dec2})$ :
\begin{equation}
(p_1-1)(p_2-1)(p_3-1) = +p_1p_2p_3 - p_1p_2 - p_2p_3 - p_3p_1 + p_1 + p_3 + p_3 -1 \label{dec3}
\end{equation}
which will still be widespread.

It is always a question of placing the modular elements linked to the $+$ signs in the numerator of the modular unit searched and those linked to the $-$ sign in the denominator.
\ms

On a le résultat suivant :
\ms

\begin{Thm}\label{Thp1p2p3}
Let $p_1$, $p_2$ and $p_3$ be distinct prime numbers, $(r_1,r_2,r_3)\in \N^{*3}$.

For $N = p_1^{r_1}p_2^{r_2}p_3^{r_3}$, the function
\begin{equation}
\Delta_{N}(\tau) = \frac{\eta_{p_1}\, \eta_{p_2}\, \eta_{p_3}\, \eta_{p_1p_2p_3}} {\eta_1\, \eta_{p_1p_2}\, \eta_{p_2p_3}\, \eta_{p_1p_3}}\, \pa{\frac{N\tau}{p_1p_2p_3}} \in M_{\frac{1}{2} (p_1-1)(p_2-1)(p_3-1)} (\Gamma_0(N)). \label{p1p2p3}
\end{equation}
is a strong $N$ level modular unit.
\end{Thm}
\ms

For example, when $N = 60 = 2^2.3.5$ :

\begin{equation}
\Delta_{60}(\tau) = q^{48}\prod_{n=1}^{+\infty} \frac{(1-q^{4n})^2 (1-q^{6n})^3 (1-q^{10n})^5}{(1-q^{2n}) (1-q^{12n})^6 (1-q^{20n})^{10} (1-q^{30n})^{15} (1-q^{60n})^{30}}\\
\end{equation}
which is in $M_{4}(\Gamma_0(60))$.
\ms

As can be seen from Theorems I-\ref{Thp1} and I-\ref{Thp1p2}, the exceptional $p=2$ and $p=3$ cases gradually faded to give a unified result on Theorem I-\ref{Thp1p2p3} that we now establish in the general case of a number of prime factors greater than or equal to $3$.
\bs
\ms

\item[$\bullet$] \textbf{The general case}
\ms

Given corollary I-\ref{Dilat2}, the search for a strong modular unit at level $N = p_1^{r_1}\ldots p_n^{r_n}$, $n\se 3$ is reduced to that of a strong modular unit at level $N = p_1\ldots p_n$.
\ms

For $N = p_1\ldots p_n$, $n\se 3$ and the prime numbers $p_i$ all distinct, we consider equality which is finally directly related to what will be the weight of $\Delta_N$ :
\begin{equation}
\begin{array}{lcl}
(p_1-1)(p_2-1)\ldots (p_n-1) & = & + \ds \prod_{1\ie i \ie n} p_i\\
 &   & - \ds \sum_{1\ie k_1 \ie n} \pa{\prod_{1\ie i \ie n, \ i \not=k_1} p_i}\\
 &   & \ds + \sum_{1\ie k_1 < k_2 \ie n} \pa{\prod_{1\ie i \ie n, \ i \not\in\{k_1,k_2\}} p_i}\\
 &   & \vdots\\
 &   & \ds +(-1)^{n-1} \sum_{1\ie i \ie n} p_i\\
 &   & \ds +(-1)^{n}.
\end{array}
\label{decN}
\end{equation}
\ms

The $m$ divisors of $p_1\ldots p_n$ preceded by a + sign will provide a $\eta_m$ term to the numerator, those preceded by a - sign will provide a $\eta_m$ term to the denominator. Let's formalize that.
\ms

The divisors of $p_1\ldots p_n$ are exactly the elements of
\begin{equation}
\Omega = \{p_1^{\eps_1}\ldots p_n^{\eps_n}, \ (\eps_1,\ldots,\eps_n)\in \{0,1\}^n\}.
\end{equation}

For $m = p_1^{\eps_1}\ldots p_n^{\eps_n} \in \Omega$, note $\ds \alpha_m = \alpha(\eps) = (-1)^{n - \sum_{i=1}^n \eps_i}$, the final result is then the following.
\ms

\begin{Thm}\label{Thp1p2p3p4}
Let be $p_1,$\ldots, $p_n$ prime numbers all distinct with $n\se 3$. So
\begin{equation}
\begin{array}{lcl}
\Delta_{p_1\ldots p_n}(\tau) & = & \ds \prod_{m\in \Omega} \eta_m(\tau)^{\alpha_m} \vspace{1mm}\\
& = & \ds \prod_{\eps\in \{0,1\}^n} \eta(p_1^{\eps_1}\ldots p_n^{\eps_n}\tau)^{\alpha(\eps)p_1^{\eps_1}\ldots p_n^{\eps_n}}
\end{array}
\end{equation}

is a strong modular unit with respect to $\Gamma_0(p_1\ldots p_n)$ with $\ds \Delta_{p_1\ldots p_n} \in M_{\frac{1}{2} (p_1-1)\ldots(p_n-1)} (\Gamma_0(p_1\ldots p_n))$.
\ms

In the general case where $N = p_1^{r_1}\ldots p_n^{r_n}$, $n\se 3$, 
\begin{equation}
\Delta_N(\tau) = \Delta_{p_1\ldots p_n}\pa{\frac{N\tau}{p_1\ldots p_n}}
\end{equation}
is a strong modular unit with respect to $\Gamma_0(N)$ with $\ds \Delta_{N} \in M_{\frac{1}{2} (p_1-1)\ldots(p_n-1)}(\Gamma_0(N))$.
\end{Thm}
\ms

\begin{proof}
With the corollary I-\ref{Dilat2}, it is enough to establish that $\Delta_N$ verifies the assumptions of Theorem I-\ref{ThUMF} when $N=p_1\ldots p_n$, $n\se 3$. For $m=p_1^{\eps_1} \ldots p_n^{\eps_n}\in \Omega$, we have $a_m = \alpha_m m$.
\ms

For condition $(i)$,
\begin{equation*}
\prod_{m\mid p_1\ldots p_n} m'^{a_m} = \prod_{m\mid p_1\ldots p_n} m^{a_{m'}} =  \prod_{\eps\in \{0,1\}^n} [p_1^{\eps_1} \ldots p_n^{\eps_n}]^{(-1)^{(\sum_{i=1}^n \eps_i)}p_1^{1-\eps_1} \ldots p_n^{1-\eps_n}} = \prod_{i=1}^n p_i^{\delta_i}.
\end{equation*}

It is about showing that the $\delta_i$ are even integers. By symmetry, let's show this for $i=n$.

We consider the $2^{n-1}$ divisors of $p_1\ldots p_n$ which contain $p_n$, that is $m =p_1^{\eps_1} \ldots p_{n-1}^{\eps_{n-1}} p_n$.

We find
\begin{equation*}
\begin{array}{lcl}
\delta_n & = & \ds \sum_{\eps\in \{0,1\}^{n-1}} (-1)^{(1+\sum_{i=1}^{n-1} \eps_i)} p_1^{1-\eps_1} \ldots p_{n-1}^{1-\eps_{n-1}}\\
& = & \ds (-1)^n \sum_{\eps\in \{0,1\}^{n-1}} (-1)^{\sum_{i=1}^{n-1} \eps_i} p_1^{\eps_1} \ldots p_{n-1}^{\eps_{n-1}}\\
& = & \ds (-1)^n \prod_{j=1}^{n-1} (1-p_j).
\end{array}
\end{equation*}

Generally speaking, $\delta_i = \ds -\prod_{1\ie j\ie n, j\not= i} (p_j-1)$ which is even, even in the case where $2$ is one of the primary factors since $n\se 3$.
\ms

For condition $(ii)$,
\begin{equation*}
\begin{array}{lcl}
\ds \frac{1}{24}\sum_{m\mid p_1\ldots p_n} m a_m  & = & \ds \frac{1}{24} \sum_{\eps\in \{0,1\}^n} (-1)^{n-\sum_{i=1}^n \eps_i} p_1^{2\eps_1} \ldots p_n^{2\eps_n} \\
& = & \ds \frac{(-1)^n}{24} \sum_{\eps\in \{0,1\}^n} (-p_1^{2})^{\eps_1} \ldots (-p_n^{2})^{\eps_n}\\
& = & \ds \frac{1}{24} \prod_{i=1}^n (p_i^2-1)
\end{array}
\end{equation*}
which is a positive integer because one of the factors, let's say $p$, is greater than or equal to $5$, which results in $ \frac{1}{24} (p^2-1)\in \N^*$ and the result.\ms

For condition $(iii)$ two cases are distinguished when $c\in \llbracket1,p_1\ldots p_n - 1\rrbracket$.
\ms

$\bullet$ If $gcd(c,p_1\ldots p_n) = 1$.

\begin{equation*}
\begin{array}{lcl}
\ds \sum_{m\mid p_1\ldots p_n} \frac{gcd(c,m)^2}{m} a_m  & = & \ds \sum_{\eps\in \{0,1\}^n} (-1)^{n-\sum_{i=1}^n \eps_i} \\
& = & \ds (-1)^n \prod_{i=1}^n (1-1) = 0.
\end{array}
\end{equation*}
\ms

$\bullet$ If $gcd(c,p_1\ldots p_n) > 1$.

If we reindex the sequence $p_1,\ldots, p_n$, we write $c = p_1\ldots p_r c'$ with $0 < r < n$ and $gcd(c',p_1\ldots p_n) = 1$.

\begin{equation*}
\begin{array}{lcl}
\ds \sum_{m\mid p_1\ldots p_n} \frac{gcd(c,m)^2}{m} a_m  & = & \ds \sum_{\eps\in \{0,1\}^n} gcd(p_1\ldots p_r,p_1^{\eps_1}\ldots p_n^{\eps_n})^2 \  \alpha(p_1^{\eps_1}\ldots p_n^{\eps_n}) \\
& = & \sum_{\eps\in \{0,1\}^n} (-1)^{n-\sum_{i=1}^n \eps_i} p_1^{2\eps_1}\ldots p_r^{2\eps_r} \\
& = & 0
\end{array}
\end{equation*}

because if you group the terms $\eps_n = $0 and $\eps_n = $1, you get two opposite terms.
\end{proof}
\ms

The simplest example with four prime factors is $N = 210 = 2.3.5.7$:
\begin{equation}
\begin{array}{lcl}
\Delta_{210}(\tau) & = & \ds q^{1152}\prod_{n=1}^{+\infty} \frac{(1-q^n) (1-q^{6n})^6 (1-q^{10n})^{10} (1-q^{14*n})^{14}} {(1-q^{2n})^2 (1-q^{3n})^3 (1-q^{5n})^5 (1-q^{7n})^7}\\
 &   & {\hskip 3cm} \ds \frac{(1-q^{15n})^{15} (1-q^{21n})^{21} (1-q^{35n})^{35} (1-q^{210n})^{210}} {(1-q^{30n})^{30} (1-q^{42n})^{42} (1-q^{70n})^{70} (1-q^{105n})^{105}}
\end{array}
\end{equation}
which is the minimum strong modular unit at level $210$. It is in modular space $M_{24}(\Gamma_0(210))$.
\ei
\bs

\begin{Not}
We will write down $\rho_N$ the weight of $\Delta_N$.
\end{Not}

Let $p$, $p_1$\ldots \ be distinct prime numbers and $r$, $r_1$\ldots \ positive integers.

The following table summarizes the characteristics of $\Delta_N$.
\ms

\begin{equation*}
\begin{array}{|l|l|l|l|l|}
\hline
N & \rho_N & \nu(\Delta_N) & \Delta_N(\tau) \\
\hline

2  & \ds 4 & \ds 1 & \ds \eta(\tau)^{-8} \eta(2\tau)^{16} \\
\hline

4 & \ds 2 & \ds 1 & \ds \eta(2\tau)^{-4} \eta(4\tau)^{8} \\
\hline

2^r, \ r\se 2 & \ds 2 & \ds 2^{r-2} & \ds \Delta_4(2^{r-2}\tau) \\
\hline

3 & \ds 6 & \ds 2 & \ds \eta(\tau)^{-6} \eta(3\tau)^{18} \\ 
\hline

9 & \ds 2 & \ds 2 & \ds \eta(3\tau)^{-2} \eta(9\tau)^{6} \\
\hline

3^r, \ r\se 2 & \ds 2 & \ds 2.3^{r-2} & \ds \Delta_9(3^{r-2}\tau) \\
\hline

p \se 5 & \ds p-1 & \frac{1}{12}(p^2 - 1) & \ds \eta(\tau)^{-2} \eta(p\tau)^{2p} \\ 
\hline

p^r, \ r\se 1 & \ds p-1 & \frac{1}{12}p^{r-1} (p^2 - 1)& \ds \Delta_p(p^{r-1}\tau) \\
\hline

2p & \ds p-1 & \frac{1}{4}(p^2 - 1) & \ds \eta(\tau)^{2} \eta(2\tau)^{-4} \eta(p\tau)^{-2p} \eta(2p\tau)^{4p} \\
\hline

2^{r_1}p^{r_2} & \ds p-1 & 2^{r_1-3}p^{r_2-1}(p^2 - 1) & \ds \Delta_{2p}(2^{r_1-1}p^{r_2-1}\tau) \\
\hline

p_1p_2, \ p_1, p_2\se 3 & \frac{1}{2}(p_1-1)(p_2-1) & \frac{1}{24}(p_1^2-1)(p_2^2-1) & \ds \eta(\tau) \eta(p_1\tau)^{-p_1} \eta(p_2\tau)^{-p_2} \eta(p_1p_2\tau)^{p_1p_2} \\
\hline

p_1^{r_1}p_2^{r_2}, \ p_1, p_2\se 3 & \frac{1}{2}(p_1-1)(p_2-1) & \frac{p_1^{r_1-1}p_2^{r_2-1}}{24}(p_1^2-1)(p_2^2-1) & \ds \Delta_{p_1p_2}(p_1^{r_1-1}p_2^{r_2-1}\tau) \\
\hline

p_1\ldots p_n, \ n\se 3 & \frac{1}{2} (p_1-1)\ldots(p_n-1) & \frac{1}{24} (p_1^2-1)\ldots(p_n^2-1)  & \ds \prod_{\eps\in \{0,1\}^n} \eta(p_1^{\eps_1}\ldots p_n^{\eps_n}\tau)^{\alpha(\eps)p_1^{\eps_1}\ldots p_n^{\eps_n}} \\
\hline

p_1^{r_1}\ldots p_n^{r_n}, \ n\se 3 & \frac{1}{2} (p_1-1)\ldots(p_n-1) & \frac{p_1^{r_1-1}\ldots p_n^{r_n-1}}{24} {\scriptstyle (p_1^2-1)\ldots(p_n^2-1)} & \ds \Delta_{p_1\ldots p_n}(p_1^{r_1-1}\ldots p_n^{r_n-1}\tau) \\

\hline

\end{array}
\end{equation*}

With $\eps = (\eps_1,\ldots,\eps_n)\in \{0,1\}^n$ and $\alpha_n(\eps) = (-1)^{n - \sum_{i=1}^n \eps_i}$
\bs
\bs

\section{-- Structure and bases of $\pa{M_{2k}(\Gamma_0(N))}_{k\in \N^*}$, $N$ positive integer}
\ms

Write the explicit formula giving the dimension of $M_{2}(\Gamma_0(N))$ in the general case, we will consult again \cite{Miya}, \cite{Diam} and \cite{Stein}.
\ms

For $p$ prime and $N\in \N^*$, we note $\nu_p(N)$ the power of $p$ in the decomposition of $N$ into prime factors. One then poses

\begin{equation}
\begin{array}{rcl}
\ms

\ds \mu_0(N) & = & \ds \prod_{p\mid N} \pa{p^{\nu_p(N)} + p^{\nu_p(N)-1}},\\
\ms

\ds \mu_{0,2}(N) & = & \ds
\left\{ \begin{array}{ll}
0 & \text{si} \ 4\mid N,\\
\prod_{p\mid N} \pa{1+ \pa{\frac{-4}{p}}} & \text{sinon},
\end{array}\right. \\
\ms

\ds \mu_{0,3}(N) & = & \ds
\left\{ \begin{array}{ll}
0 & \text{si} \ 2\mid N \ \text{ou} \ 9\mid N,\\
\prod_{p\mid N} \pa{1+ \pa{\frac{-3}{p}}} & \text{sinon},
\end{array}\right. \\
\ms

c_0(N) & = & \ds \sum_{d\mid N} \varphi(gcd(d,N/d)),\\
\ms

g_0(N) & = & \ds 1 +\frac{\mu_0(N)}{12} - \frac{\mu_{0,2}(N)}{4} - \frac{\mu_{0,3}(N)}{3} - \frac{c_0(N)}{2}.
\end{array}
\end{equation}
\ms

The $M_{2k}(\Gamma_0(N))$ space is divided into the cuspidal subspace $S_{2k}(\Gamma_0(N))$ and the Eisenstein subspace $E_{2k}(\Gamma_0(N))$:

\begin{equation}
M_{2k}(\Gamma_0(N)) = S_{2k}(\Gamma_0(N)) \oplus E_{2k}(\Gamma_0(N)).
\end{equation}

We have the relationships:
\begin{equation}
\dim(S_{2k}(\Gamma_0(N))) = \ds
\left\{\begin{array}{l}
\vspace{2mm}
g_0(N) \ \ \text{if \ k=1}, \\

(2k-1)(g_0(N)-1) + \pa{k-1} c_0(N) + \mu_{0,2}(N)\lfloor\frac{k}{2}\rfloor + \mu_{0,3}(N)\lfloor\frac{2k}{3}\rfloor \ \ \text{si} \ k\se 2,
\end{array}\right.
\label{DimM20}
\end{equation}


\begin{equation}
\dim(E_{2k}(\Gamma_0(N))) = \left\{
\begin{array}{ll}
\vspace{2mm}
c_0(N) & \text{if} \ k\not = 1\\
c_0(N)-1 & \text{if} \ k=1.
\end{array}
\right.
\end{equation}

Hence finally
\begin{equation}
\forall k\in\N^*, \ \ \dim(M_{2k}(\Gamma_0(N))) = \ds (2k-1)(g_0(N)-1) + kc_0(N) + \mu_{0,2}(N)\lfloor\frac{k}{2}\rfloor + \mu_{0,3}(N)\lfloor\frac{2k}{3}\rfloor
\label{DimM1}
\end{equation}

\begin{Thm}\label{LemDim2}
Let's say $N \in\N^*$ and $\rho_N$ the weight of $\Delta_N$. For any $k\in \N^*$,
\begin{equation}
\dim(M_{2k+\rho_N}(\Gamma_0(N))) - \dim(M_{2k}(\Gamma_0(N))) = \nu(\Delta_N). \label{Dimg}
\end{equation}
\end{Thm}
\ms

This result generalizes, as announced, lemma I-\ref{LemDim1}.

When $N=1$, the weight of $\Delta_1 = \Delta$ is $12$ and its valuation $1$, $(\ref{Dimg})$ is classically verified.
\ms

To establish the result, we can use directly the equation $(\ref{DimM1})$, but that requires to study several cases according to the divisibility of $\rho_N$ by $3$ and $4$. To avoid this, let us establish the following result, analogous to corollary I-\ref{CorVal1}.

\begin{Lem}\label{LemVal2}
Let $N$ be an integer greater than or equal to $2$.
If $(E_{2k,N}^{(r)})_{0\ie r \ie d_{2k}(N)-1}$ is an upper triangular basis of $M_{2k}(\Gamma_0(N))$, then $\nu(E_{2k,N}^{(1)}) = $1 for every $k\se 2$.
\end{Lem}

\begin{proof}
Let us take again some results of the demonstration of Theorem I-\ref{ThmVal1}.

\begin{equation*}
\begin{array}{lcl}
G_2(\tau) & = & \ds \sum_{m\in \Z}\sum_{n\in \Z'_m} \frac{1}{(m\tau+n)^{2}}\\
& = & \ds 2\zeta(2) - 8\pi^2\sum_{n=1}^{+\infty} \sigma(n) q^n = \ds 2\zeta(2) - 8\pi^2 q + O(q^2)
\end{array}
\end{equation*}
where $\Z'_m = \Z-\{0\}$ si $m=0$ and $\Z'_m = \Z$ otherwise.

It is known that $H_{2,N}(\tau) = \frac{1}{2(1-N)\zeta(2)}\pa{G_2(\tau)-NG_2(N\tau)} = 1 - \frac{24}{(1-N)} q + O(q^2)$ belong to $M_2(\Gamma_0(N))$.

It can then be seen that $H_{2,N}(\tau)^2 = 1 - \frac{48}{(1-N)} q + O(q^2)$ belong to $M_4(\Gamma_0(N))$.

The Eisenstein series $E_4(\tau) = 1 + 240q + O(q^2)$ also belongs to $M_4(\Gamma_0(N))$, and consequently $E_4 - H_{2,N}^2 = \pa{240 + \frac{48}{(1-N)}} q + O(q^2)$ also belongs to $M_4(\Gamma_0(N))$ with valuation $1$.
\ms

With the usual notations, $(E_4 - H_{2,N}^2)[E_{2,N}{0}]^{2k-4}$ is an element of $M_{2k}(\Gamma_0(N))$ with valuation $1$ for every $k\se 2$. The result can be deduced from this.
\end{proof}
\ms

\begin{proof}[Demonstration of Theorem I-\ref{LemDim2}]
$ $

We note again $(E_{2k,N}^{(r)})_{0\ie r \ie d_{2k}(N)-1}$ a unitary upper triangular basis of $M_{2k}(\Gamma_0(N))$
\ms

We deduce from Lemma I- \ref{LemVal2} the equalities $\nu(E_{4,N}^{0}) = 0$ and $\nu(E_{4,N}^{1}) = 1$.

As big as you take an integer $a$, you find that for $\ell \in \llbracket 0,a\rrbracket$, the modular form $[E_{2,2}^{(0)}]^{2a-2\ell}[E_{4,N}{1}]^{\ell}$ belongs to $M_{4a}(\Gamma_0(N))$ with valuation $\ell$.

We then choose $a = \nu(\Delta_N)$, and therefore a unitary upper triangular basis of $M_{4a}(\Gamma_0(N))$ does not show any jump in its first $a$ items. This remains true for $M_{2k}(\Gamma_0(N))$, whatever $k\se 2a$ : just multiply the first $a$ of the $M_{4a}(\Gamma_0(N))$ basis by $[E_{2,2}^{(0)}]^{k-2a}$ to see it.
\ms

Finally, using Theorem I-\ref{ThmStruct1},

\begin{equation*}
\forall h\in \N, \ \ M_{4a+2h}(\Gamma_0(N)) = \Phi.M_{2h-\rho_N}(\Gamma_0(N)) \oplus Vect\pa{E_{4a+2h,N}^{(s)} \ / \ \nu(E_{4a+2h,N}^{(s)}) < \nu(\Delta_N)}.
\end{equation*}

Now, we just established that for $h\se 0$,

$\ac{E_{4a+2h,N}^{(s)} \ / \ \nu(E_{4a+2h,N}^{(s)}) < \nu(\Delta_N)} = \ac{E_{4a+2h,N}^{(s)} \ / \ 0 \ie s < \nu(\Delta_N)}$ which is cardinal $\nu(\Delta_N)$. The relationship $(\ref{Dimg})$ is therefore established for all $k\se k_0$, $k_0$ being fixed.
\ms

We then observe, thanks to $(\ref{DimM1})$, that $k\mapsto \dim(M_{2k+\rho_N}(\Gamma_0(N))) - \dim(M_{2k}(\Gamma_0(N))$ is a $6$ period function starting from $k=1$. Since it is constant from $k_0$, it is a constant function for all $k\se 1$, necessarily equal to $\nu(\Delta_N)$, it is the announced result.
\end{proof}
\ms

Moreover, we deduce from Theorem I-\ref{ThmStruct1} the equality
\[\forall k\in \N, \ \ \dim(M_{2k+\rho_N}(\Gamma_0(N))) = \dim(M_{2k}(\Gamma_0(N))) + Card(\{s \ / \ \nu(E_{2k+\rho_N,N}^{(s)}) < \nu(\Delta_N)\}).\]

As a consequence
$\ds Card(\{s \ / \ \nu(E_{2k+\rho_N,N}^{(s)}) < \nu(\Delta_N)\}) = \nu(\Delta_N)$  for $k \se $1 and the next result.

\begin{Thm}
Let be $N\in\N^*$ and, for the whole $k\se 1$, $(E_{2k,N}^{(r)})_{0\ie r \ie d_{2k}(N)-1}$ a unitary upper triangular basis of $M_{2k}(\Gamma_0(N))$. So
\begin{equation}
\forall k \se \frac{\rho_N}{2}+1, \ \ \forall r\in \llbracket0,\nu(\Delta_N) -1\rrbracket, \ \ \nu(E_{2k,N}^{(r)}) = r.
\end{equation}
In addition, you can choose
\begin{equation}
\forall k \se \frac{\rho_N}{2}+1, \ \ \forall r\in \llbracket0,\nu(\Delta_N) -1\rrbracket, \ \ E_{2k,N}^{(r)} = E_{\rho_N+2,N}^{(r)}[E_{2,N}^{(0)}]^{k-\frac{\rho_N}{2}-1}.
\end{equation}
\end{Thm}
\bs

The theorem of structure and construction of bases can be formulated in its final form.

\begin{Thm}\label{ThmStruct10}
Let $N$ be a positive integer, then 
\begin{equation}
\forall k\in \N, \ k\se \frac{\rho_N}{2}, \ \ M_{2k}(\Gamma_0(N)) = \Delta_N.M_{2k-\rho_N}(\Gamma_0(N)) \oplus Vect\pa{E_{\rho_N+2,N}^{(s)}[E_{2,N}^{(0)}]^{k-\frac{\rho_N}{2}-1} \ / \ 0\ie s < \nu(\Delta_N)}.
\end{equation}
Therefore, if $k\in \N^*$ and $k = q\frac{\rho_N}{2}+r$ with $1\ie r \ie \frac{\rho_N}{2}$,
\begin{equation}
 M_{2k}(\Gamma_0(N)) = \Delta_N^{q}. M_{2r}(\Gamma_0(N)) \bigoplus_{n=0}^{q-1}  \Delta_N^n .Vect \pa{E_{\rho_N+2,N}^{(s)}[E_{2,N}^{(0)}]^{k-(n+1)\frac{\rho_N}{2}-1} \ / \ 0\ie s < \nu(\Delta_N)}.
\end{equation}
\end{Thm}

In practice, there are simple strategies for determining elements of a ${\cal B}_{2k+2}(\Gamma_0(N))$ basis knowing ${\cal B}_{2k}(\Gamma_0(N))$. For example :

\be
\item[-]
We can impose, as in Theorem I-\ref{ThmStruct10}, $E_{2k+2,N}^{(s)} = E_{2k,N}^{(s)}E_{2,N}^{(0)}$ for $0\ie s \ie d_{2k}(N)-1$, and also $E_{2k+2,N}^{(d_{2k}(N))} = E_{2k,N}^{(d_{2k}(N))}E_{2,N}^{(1)}$, for instance.

\item[-]
When $d_{\rho_N}(N) = \nu(\rho_N)+1$, for $r\in \llbracket0,\nu(\Delta_N) -1\rrbracket$ we have $\nu(E_{\rho_N,N}^{(r)}) = r$, and we can choose
\begin{equation}
\forall k \se \frac{\rho_N}{2}, \ \ \forall r\in \llbracket0,\nu(\Delta_N) -1\rrbracket, \ \ E_{2k,N}^{(r)} = E_{\rho_N,N}^{(r)}[E_{2,N}^{(0)}]^{k-\frac{\rho_N}{2}}.
\end{equation}

Theorem I-\ref{ThmStruct10} is then slightly improved since the knowledge of the $({\cal B}_{2k}(\Gamma_0(N)))_{1\ie k\ie \frac{\rho_N}{2}}$ bases is sufficient to generate a total sequence of bases $({\cal B}_{2k}(\Gamma_0(N)))_{k\in \N}$.
\ms

This remark will be exploited in Part II because $d_{\rho_N}(N) = \nu(\rho_N)+1$ is checked when $1\ie N\ie 10$.
\ee
\ms

Regarding the calculation of a unitary upper triangular basis ${\cal B}_{2k}(\Gamma_0(N))$, $k\se 1$, Theorem I-\ref{ThmStruct10} is operational since the knowledge of all bases is reduced to that of the finite sequence of bases $({\cal B}_{2k}(\Gamma_0(N)))_{1\ie k\ie \frac{\rho_N}{2}+1}$. We will see in the second part of this article that this is possible using elliptic functions \cite{FeauE}. We will explicitly provide such bases for $1 \ie N \ie 10$.

Moreover, as we have seen, the knowledge of unitary upper triangular bases $({\cal B}_{2k}(\Gamma_0(N)))_{1\ie k\ie k_0}$, for some $k_0 \ie \frac{\rho_N}{2}$ enables one to obtain many elements of ${\cal B}_{2k_0+2}(\Gamma_0(N))$. For example, we might notice that $E_{2,N}^{0}{\cal B}_{2k_0}(\Gamma_0(N)) \subset {\cal B}_{2k_0+2}(\Gamma_0(N))$, which greatly reduces the number of new modular forms to be determined to obtain a ${\cal S}_{2k_0+2}(\Gamma_0(N))$ unitary upper triangular basis.
\ms

On the other hand, SAGE allows explicit calculation of elements of $({\cal B}_{2k}(\Gamma_0(N)))_{1\ie k \ie \frac{\rho_N}{2}+1}$ to a given precision, which then gives access, without more complex calculations, to all bases $({\cal B}_{2k}(\Gamma_0(N))_{k\in \N^*}$ with the same precision.
\bs
\bs

\section{ -- From theory to practice}
\ms

Theorem I-\ref{ThmStruct10} reveals the structure of classical modular form spaces with respect to $\Gamma_0(N)$. To obtain unitary upper triangular bases of these spaces, it remains to determine unitary upper triangular bases ${\cal B}_{2k}(\Gamma_0(N)) = (E_{2k,N}^{(s)})_{0\ie s \ie d_{2k}(N)-1}$, $1 \ie k \ie \frac{\rho_N}{2}$ as well as the elements of ${\cal B}_{\rho_N+2}(\Gamma_0(N))$ : $(E_{\rho_N+2,N}^{(s)})_{0\ie s \ie \nu(\Delta_N)-1}$.
\ms

It is not an easy task, but many modular forms are identified in the literature, one can consult for example \cite{Kohl} for a broad study on the subject.
\ms

It is with this in mind that the second part of this article was written.  Show that, for $N$ levels between $1$ and $10$, it is possible to explain these first bases that initiate a structured sequence of bases for a given $N$ level.
\ms

It must nevertheless be noted that, despite generic processes that will certainly identify certain elements of these bases through elliptic functions, these will be case-by-case studies. Hecke's operator theory can also provide basic elements in a generic way \cite{Stein}.
\ms

However, the calculation approach can directly benefit from the results of the previous sections. Knowing the basics for $1 \ie 2k \ie \rho_N+2$ at an accuracy of $O(q^m)$ allows you to simply obtain upper triangular basics for any weight $2k > \rho_N+2$ with the same accuracy. It is thanks to Theorem I-\ref{ThmStruct10}, as well as the explicit formulas established for $\Delta_N$ in this part I and for $E_{2,N}^{(0)}$ in part II.
\bs
\bs

\section{ -- Conclusion of Part I}
\ms

To close this part, a few words to better situate the $\Delta_N$ functions compared to previous works. The products and quotients of $\eta$ functions have been studied by Rademacher \cite{RadRam} who introduces the functions $\ds \phi_\delta(\tau) = \eta(\delta\tau)/\eta(\tau)$ to establish that, if $p \se 5$ is prime and $r$ is an even integer, then $\phi_p^r$ is a weakly modular function of weight $0$ with respect to $\Gamma_0(p)$. This result was generalized by Newmann \cite{New1}, \cite{New2} to construct, still from the $\phi_\delta$ functions, weakly modular functions with respect to $\Gamma_0(N)$, with any $N$ this time and always of weight $0$.
\ms

Theorem I-\ref{ThUMF}, which enables one to establish that $\Delta_N$ functions are strong modular units, was essentially obtained by Ligozat \cite{Ligo} when studying elliptical modular curves. From then on, we essentially looked for $\eta$-quotients in order to find cuspidal modular forms. This is why, in our opinion, the notion of strong modular unity does not seem to have emerged, eclipsed by the highly justified importance taken by the cuspidal forms following Hecke's founding work.
\bs

By introducing the  $\Delta_N$ functions, we were able to elucidate the structure of the modular space sequences $(M_{2k}(\Gamma_0(N)))_{k\in\N^*}$, and provide an operational tool to explain a basis for each of these spaces. In the second part of this work, we will apply Theorem I-\ref{ThmStruct10} to the explicit search for bases when $1\ie N\ie 10$. To determine the few missing modular forms, we will introduce renormalized elliptic functions completely adapted to the search for modular forms \cite{FeauE}.
\bs

\end{document}